\RequirePackage{tikz}
\documentclass[11pt]{article}
\usepackage[utf8]{inputenc}

\usepackage{amsfonts}
\usepackage{amsmath}
\usepackage{amssymb}
\usepackage{amsthm}
\usepackage{fullpage}
\usepackage{fancyhdr}
\usepackage{microtype}
\usepackage{lastpage}
\usepackage{lipsum}
\usepackage{mathtools}
\usepackage{xcolor}
\usepackage{ bbold }

\usepackage{caption}
\usepackage{subcaption}
\usepackage{tikz-cd}
\usepackage{bbm}
\tikzcdset{scale cd/.style={every label/.append style={scale=#1},
    cells={nodes={scale=#1}}}}

\everymath{\displaystyle}
\title{Combinatorial Methods for Barcode Analysis}
\author{Edgar Jaramillo-Rodriguez$^1$}
\date{$^1$ Department of Mathematics, University of California, Davis. \\ E-mail: ejaramillo@ucdavis.edu}

\renewcommand{\epsilon}{\varepsilon}

\DeclareMathOperator{\conv}{conv}
\DeclareMathOperator{\inv}{inv}
\DeclareMathOperator{\invm}{invm}
\DeclareMathOperator{\cross}{cross\#}
\DeclareMathOperator{\image}{Im}
\DeclareMathOperator{\Id}{Id}

\makeatletter
\newtheorem*{rep@theorem}{\rep@title}
\newcommand{\newreptheorem}[2]{%
\newenvironment{rep#1}[1]{%
 \def\rep@title{#2 \ref{##1}}%
 \begin{rep@theorem}}%
 {\end{rep@theorem}}}
\makeatother

\theoremstyle{plain}
\newtheorem{theorem}{Theorem}
\newreptheorem{theorem}{Theorem}
\newtheorem{lemma}{Lemma}
\newreptheorem{lemma}{Lemma}
\newtheorem{prop}{Proposition}
\newtheorem{cor}{Corollary}
\newreptheorem{cor}{Corollary}

\theoremstyle{definition}
\newtheorem{definition}{Definition}

\theoremstyle{remark}
\newtheorem{remark}{Remark}

\begin{document}
\maketitle


\begin{abstract}
    A \emph{barcode} is a finite multiset of intervals on the real line, $B = \{ (b_i, d_i)\}_{i=1}^n$.
    Barcodes are important objects in topological data analysis, where they serve as summaries of the persistent homology groups of a filtration.
    The combinatorial properties of barcodes have also been studied, mainly in the context of interval orders and interval graphs. 
    In this paper, we define a new family of maps from the space of barcodes with $n$ bars to the permutation sets of various multisets, known as \emph{multipermutations}. These multipermutations provide new \emph{combinatorial invariants} on the space of barcodes. We then define an order relation on these multipermutations, which we show can be interpreted as a \emph{crossing number} for barcodes, reminiscent of T\'{u}ran's crossing number for graphs.
    Next, we show that the resulting posets are order-isomorphic to principal ideals of a well known poset known as the multinomial Newman lattice. Consequently, these posets form the graded face-lattices of polytopes, which we refer to as \emph{barcode lattices} or \emph{barcode polytopes}. Finally, we show that for a large class of barcodes, these invariants can provide bounds on the Wasserstein and bottleneck distances between a pair of barcodes, linking these discrete invariants to continuous metrics on barcodes.
\end{abstract}


\maketitle
\section{Introduction}\label{intro}

A \emph{barcode} is a finite multiset of intervals on the real line, $B = \{ (b_i, d_i)\}_{i=1}^n$. Barcodes are important objects in persistent homology, where they serve as summaries of the persistent homology groups of a given filtration \cite{Carlsson2005}. In this context, each bar in a barcode represents an interval of resolutions in a filtration during which a particular generator is present. By analyzing the lengths and arrangements of the different bars researchers are able to determine significant features in the filtration. This theory has been applied to the study of data, where a filtration is applied to some dataset and its persistent homology groups are used to determine features in the underlying distribution \cite{Carlsson2009TopologyAD, Barcodes_2008}.


It is often desirable to study statistics on barcodes and much work has been done in this area (see, for instance, \cite{cohen-edelsbrunner07}).
Recently, researchers have begun studying combinatorial invariants associated to barcodes. In particular,  Kanari, Garin, and Hess discovered a natural mapping between the space of barcodes with $n+1$ bars and the symmetric group $\mathfrak{S}_n$ in \cite{kanari2020trees, tressII}. In follow-up works, Br\"{u}ck and Garin refine this mapping in order to stratify the space of barcodes into regions with the same average and standard deviation in each endpoint and the same permutation types \cite{bruck2021stratifying}. The coordinates obtained through this mapping are an elegant blend of continuous and discrete invariants. 

\subsection{Our Contributions}

In this work, we define a new family of combinatorial invariants on the space of barcodes. The first of these invariants comes from constructing a map from the space of barcodes with $n$ bars to certain equivalence classes of the multinomial Newman Lattice, $L(2^n)$, whose elements are all permutations of the multiset $\{1^2,2^2, \ldots, n^2\}$. We then define a partial order on these equivalence classes and we call the resulting poset the \emph{combinatorial barcode poset}. This poset admits an elegant construction based on an original notion of the \emph{crossing number} of a barcode, reminiscent of T\'{u}ran's crossing number for graphs.

\begin{repcor}{crossing_num_cor}
Let $B = \{(b_i,d_i)\}_{i=1}^n$ be a barcode. Then the rank of of $B$ in the combinatorial barcode lattice, denoted $\rho{(g(B))}$, is given by
$$ \rho(g(B)) = \sum_{b_i < b_j} \cross(i,j).
$$
\end{repcor}

Our first main result is that the unlabeled barcode poset, is order-isomorphic to a sublattice of $L(2^n)$.

\begin{reptheorem}{lattice_thm}
The combinatorial barcode poset $(L(2^n)/\mathfrak{S}_n, \leq_u)$ is isomorphic to a principal ideal of the multinomial Newman lattice, $L(2^n)$. Consequently, $(L(2^n)/\mathfrak{S}_n, \leq_u)$ is a lattice.
\end{reptheorem}

Next, we generalize this construction by considering a family of maps from the space of barcodes to equivalence classes of the lattice $L((2^k+1)^n)$,
consisting of all permutations of the multiset $\{1^{2^k+1},2^{2^k+1}, \ldots n^{2^k+1}\}$, where $k \in \mathbb{Z}_{\geq 0}$.
This produces an entire family of new combinatorial invariants, denoted $g_k(B)$. We show that these invariants retain the graded lattice structure found in the $k=0$ case, and so we call the poset of all $g_k(B)$ the \emph{power $k$ barcode lattice}.

\begin{reptheorem}{multi-k}
The power $k$ barcode poset $(L((2^k+1)^n)/\mathfrak{S}_n, \leq_k)$ is isomorphic to a principal ideal of the multinomial Newman lattice, $L((2^k+1)^n)$. Consequently, $(L((2^k+1)^n)/\mathfrak{S}_n, \leq_k)$  is a lattice.
\end{reptheorem}
Moreover, we show that larger values of $k$ produce invariants, $g_k(B)$ which contain the invariants produced by smaller values. It follows that increasing $k$ amounts to refining the information in the invariant. We then prove that as $k$ goes to infinity, this invariant uniquely determines a large family of barcodes up to an affine transformation.
\begin{reptheorem}{affine}
Let $B,B'$ be fundamentally strict barcodes with $n$ bars, where $B =\{(b_i,d_i)\}_{i=1}^n$ and $B' =\{(b_i',d_i')\}_{i=1}^n$. Additionally, let $g_k(B)$ denote its power-$k$ invariant, and likewise for $B'$. If $g_k(B) = g_k(B')$ for all $k\in \mathbb{N}$ and the interval graph $G_B$ (equivalently $G_{B'}$) is connected, then there exist constants $\alpha>0$ and $\delta \in \mathbb{R}$ such that $B = \alpha B' +\delta$, where $ \alpha B' +\delta := \{(\alpha b_i'+ \delta,\alpha d_i' + \delta): i\in [n]\}$.
\end{reptheorem}
We also show that these invariants can provide upper bounds on the bottleneck and $q$-Wasserstein metrics ($d_\infty$ and $d_q$, respectively) between barcodes, up to an affine transformation.
\begin{reptheorem}{convergence}
Let $B,B'$ be $k$-strict barcodes with $n$ bars such that their power-$k$ invariants are equal, $g_k(B) = g_k(B')$. Suppose there exists a bar $(b_*,d_*) \in B$ (or equivalently in $B'$) which contains all others, that is to say $b_* \leq b_i$ and $d_* \geq d_i$ for all $i\in [n]$.
	Then there exist constants $\alpha > 0$ and $\delta \in \mathbb{R}$ such that
	\begin{align*}
    d_\infty(B, \alpha B' +\delta) &\leq \frac{\lvert d_*-b_*\rvert}{2^k}, \text{ and} \\
d_q(B, \alpha B' +\delta) &\leq (n-1)^{\frac{1}{q}} \frac{\lvert d_*-b_*\rvert}{2^k}.
\end{align*}
\end{reptheorem}

Finally, we show that the barcodes lattices form polytopes via an embedding into the classic permutahedron. In fact, these polytopes are a special case of \emph{Bruhat interval polytopes}, studied at length in \cite{Williams2015}. We use some results therein to compute the dimensions of our barcode polytopes (Corollary \ref{polytope_dim}).



\section{Background}\label{background}

\subsection{Barcodes}
The language of barcodes used below comes from persistence homology; see \cite{Barcodes_2008} for an introduction and details.

\begin{definition}
A \emph{barcode} is a finite multiset of intervals on the real line, $B = \{ (b_i, d_i)\}_{i=1}^n$, where necessarily $b_i < d_i$ for all $i\in [n]$. Each interval is called a \emph{bar}; its left endpoint $b_i$ is called its \emph{birth (time)} and its right endpoint $d_i$ is called its \emph{death (time)}. We denote the set of all barcodes with $n$ bars by $\mathcal{B}^n$.
\end{definition}

\begin{remark}
Readers familiar with persistent homology will notice that we suppose that the bars corresponding to essential classes have finite values instead of being half-open intervals. This is reasonable in the context of computing Rips/\v{C}ech complexes of a dataset, where there is at most a single essential class and it is assigned a finite death-time in order to display it in a diagram.
\end{remark}

Barcodes are often displayed as a stacked set of intervals above the real line as in Figure \ref{tda_barcode}. Often one will refer to this diagram itself as a barcode. Barcodes are also commonly represented as points $(b_i,d_i)$ in $\mathbb{R}^2$ in a figure known as a \emph{persistence diagram}. Figure \ref{pers_dgm} shows the persistence diagram for the barcode in Figure \ref{tda_barcode}. Note that the points in Figure \ref{pers_dgm} lie above the diagonal since we require that $b_i < d_i$ for all $i \in [n]$.

\begin{figure}[h]
     \centering
     \begin{subfigure}[b]{0.26\textwidth}
         \centering
         \includegraphics[width=\textwidth]{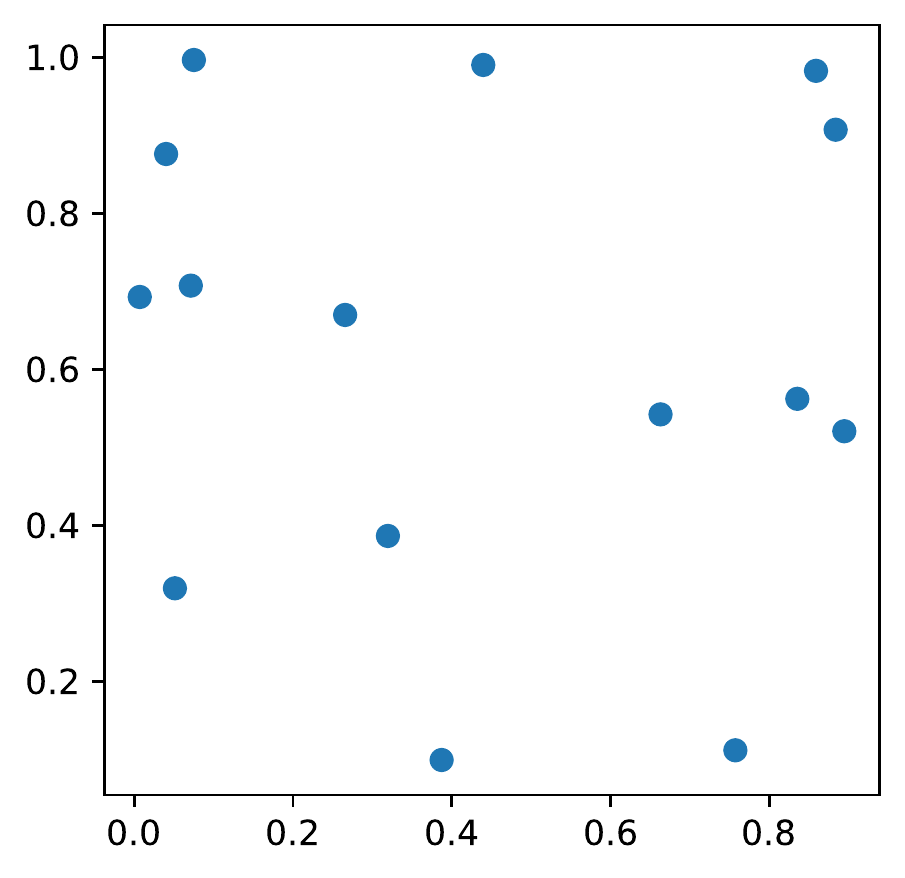}
         \caption{Sample data.}
         \label{data}
     \end{subfigure}
     \hfill
     \begin{subfigure}[b]{0.35\textwidth}
         \centering
         \includegraphics[width=\textwidth]{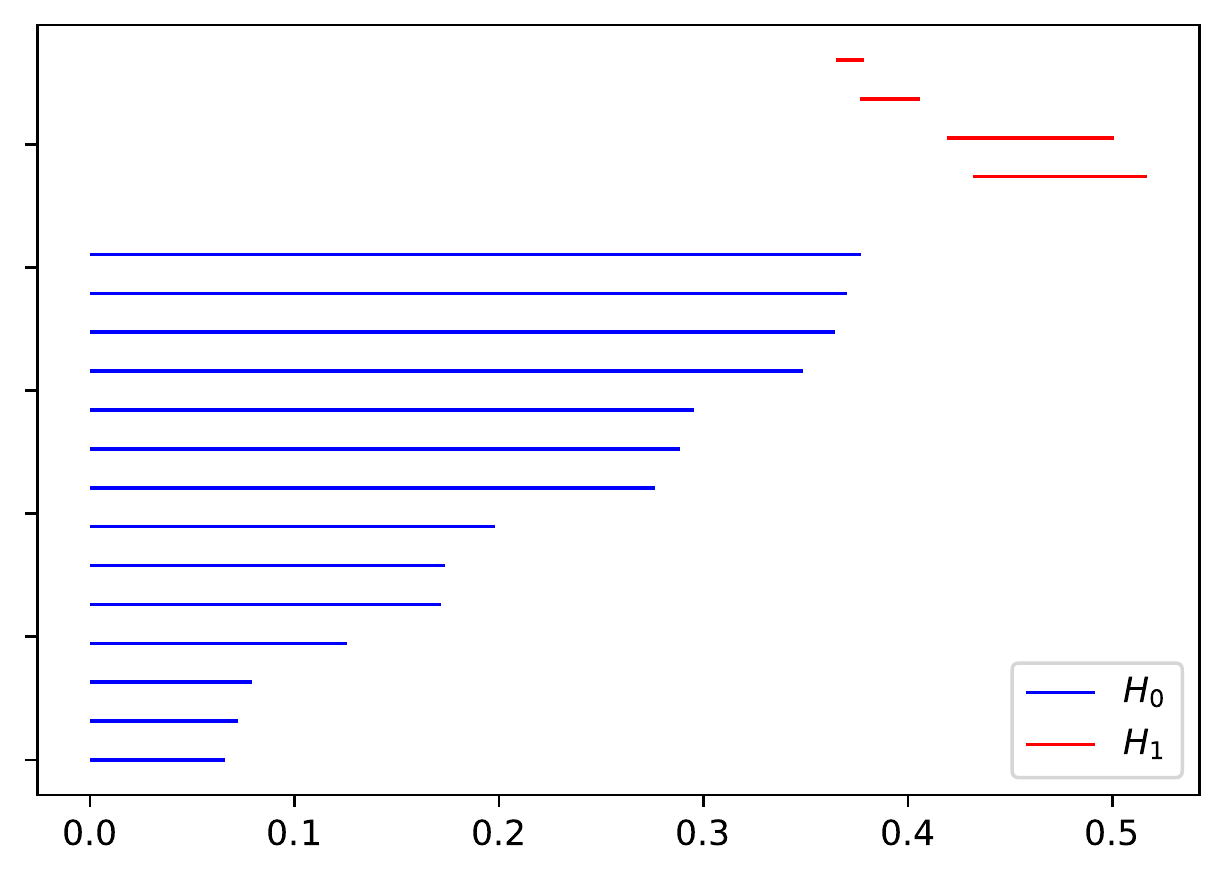}
         \caption{Persistence barcode.}
         \label{tda_barcode}
     \end{subfigure}
     \hfill
     \begin{subfigure}[b]{0.29\textwidth}
         \centering
         \includegraphics[width=\textwidth]{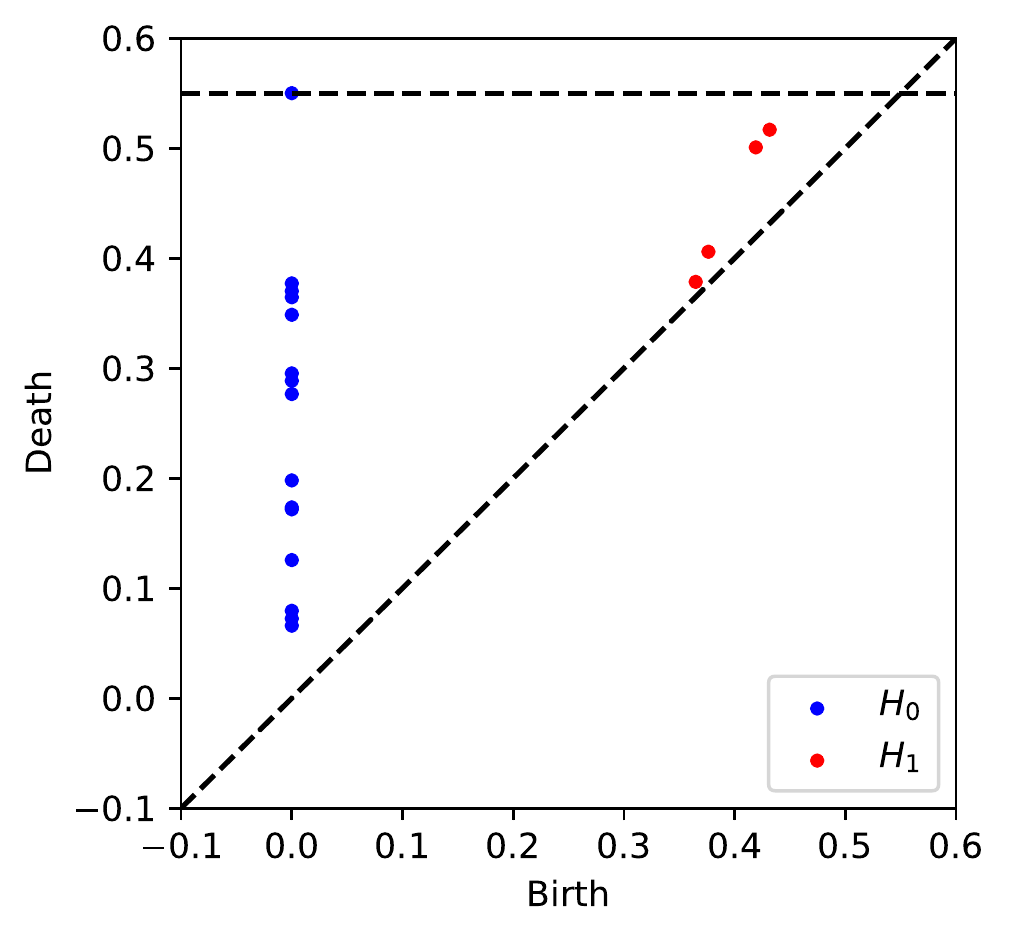}
         \caption{Persistence diagram.}
         \label{pers_dgm}
     \end{subfigure}
        \caption{We compute the Rips filtration of the sample points in (\ref{data}), then display the order 0 and order 1 persistence homology  groups as both a barcode (\ref{tda_barcode}) and a persistence diagram (\ref{pers_dgm}).}
        \label{Rips_diagrams}
\end{figure}

\begin{definition}
A barcode $B = \{ (b_i, d_i)\}_{i=1}^n$ is called \emph{strict} if $b_i \neq d_j$ for all $(i,j) \in [n]^2$ and $b_i \neq b_j$, $d_i \neq d_j$ for all $i \neq j$, i.e., if no birth/death times are repeated among the bars. We denote the set of strict barcodes with $n$ bars by $\mathcal{B}^n_{st}$.
\end{definition}

A similar definition of strict barcodes was first introduced in \cite{kanari2020trees} to define a bijection between $\mathfrak{S}_n$ and an equivalence classes of barcodes defined therein. Their definition differed slightly from ours in that: (1) we do not require that there exist  an essential bar $(b_0, d_0)$ that contains all the others, and (2) we require that no times are shared between births and deaths, $b_i \neq d_j$, which is not present in their definition. Furthermore, we emphasize that the equivalence classes defined in \cite{kanari2020trees} are different from those we will define in Section \ref{perm}.

We can define a topology on the space of all barcodes, regardless of strictness or the number of bars, by way of a distance function. Two popular and well-studied choices are the \emph{bottleneck distance}
	and the \emph{$q$-Wasserstein distance} \cite{Edel_survery2008}.
\begin{definition}\label{bottleneck}
Let $B= \{(b_i,d_i)\}_{i=1}^n$ and $B'= \{(b_i',d_i')\}_{i=1}^m$ be two barcodes. The \emph{bottleneck distance} between $B$ and $B'$ is
$$
d_\infty(B,B') = \inf_\gamma \max_{x\in B}\|x- \gamma(x)\|_\infty,
$$
where $\gamma$ runs over all perfects matchings of the points (bars) $x_i = (b_i,d_i)$ in $B$ and the points (bars) in $B'$, allowing bars to be matched to the diagonal $\Delta = \{(x,x): x\in \mathbb{R}\}$. Here $\|\cdot\|_\infty$ denotes the $\ell^\infty$-norm on $\mathbb{R}^2$.
\end{definition}
Intuitively, the bottleneck distance computes the largest distance in $\ell^\infty$-norm between a pair of matched points on the persistence diagrams of $B$ and $B'$, taking the infimum over all perfect matchings. The $q$-Wasserstein distance is similar; it can be thought of as the total distance between all matched points, again taking the infimum over all perfect matchings.
\begin{definition}\label{q-wasserstein}
Let $B= \{(b_i,d_i)\}_{i=1}^n$ and $B'= \{(b_i',d_i')\}_{i=1}^m$ be two barcodes. The \emph{$q$-Wasserstein distance} between $B$ and $B'$ is
$$
d_q(B,B') = \inf_\gamma \big( \sum_{x\in B}\|x- \gamma(x)\|_\infty^q \big)^{1/q},
$$
where $\gamma$ runs over all perfects matchings of the points (bars) $x_i = (b_i,d_i)$ in $B$ and the points (bars) in $B'$, allowing bars to be matched to the diagonal $\Delta = \{(x,x): x\in \mathbb{R}\}$. Here $\|\cdot\|_\infty$ denotes the $\ell^\infty$-norm on $\mathbb{R}^2$.
\end{definition}

\subsection{Weak Bruhat Orders and Newman Lattices}
We will assume the reader is familiar with the basic definitions and properties of posets. For a complete introduction to partially ordered sets see \cite[Ch.3]{EC1}, for example.

We recall for the reader that a \emph{lattice} is a poset $L$ for which every pair of elements has a least upper bound and greatest lower bound. 
We also recall that a \emph{principal ideal} of a lattice $L$ is a subposet of the form $I = \{s\in L: s \leq \alpha \}$ for some $\alpha \in L$; we say $I$ is the principal ideal generated by $\alpha$. It is a well known result that a principal ideal $I$ of a lattice $L$ is a sublattice of $L$.

A lattice of particular interest to us is the \emph{permutahedron} \cite{permutahedra}. Recall that for $\pi \in \mathfrak{S}_N$ an \emph{inversion} in $\pi$ is a pair $(\pi_i, \pi_j)$ such that $i <j$ and $\pi_i > \pi_j$, i.e., it is a pair of elements that appear out of order. The \emph{inversion set} of $\pi$, $\inv(\pi)$, is the set of all inversions in $\pi$. The $\emph{inversion number}$ of $\pi$ is the cardinality of its inversion set $\#\inv(\pi)$. For example, if $\pi = (1~2~5~4~3~6) \in \mathfrak{S}_6$ then $\inv(\pi) = \{(5,4), (5,3), (4,3)\}$ and $\#\inv(\pi) = 3$. Notice that $\pi \neq \sigma \implies \inv(\pi) \neq \inv(\sigma)$, so we can think of $\inv$ as an injective map from the permutations in $\mathfrak{S}_n$ to subsets of $[n]^2$.

The \emph{weak Bruhat order} (or \emph{weak order} for short) is the relation $\leq_W$ on $\mathfrak{S}_N$ defined by $\pi \leq_W \sigma$ if and only if $\inv(\pi) \subseteq \inv(\sigma)$. Note that $\pi \leq_W \sigma \implies \#\inv(\pi) \leq \#\inv(\sigma)$, but the converse need not hold.
One can show that the $\pi \lessdot_W \sigma$ if and only if $\#\inv(\pi) +1 = \#\inv(\sigma)$ and $\sigma = (i \,\,\, i+1)\pi$ for some $i \in [n-1]$, which means that $\sigma$ equals $\pi$ after transposing a pair of its \emph{adjacent} entries.

It is a well known result that the weak order on $\mathfrak{S}_n$ forms the face lattice of a polytope known as the permutahedron $\cite{Berge, permutahedra}$; therefore we to refer to the poset $(\mathfrak{S}_n \leq_W)$ as the permutahedron as well.
The weak Bruhat order can also be defined similarly on arbitrary Coxeter systems (see \cite{Coxeter}), but for this manuscript the definition on the symmetric group is sufficient.

One can generalize the construction of the permutahedron to multiset permutations. For $\mathbf{m}= (m_1, \ldots,m_n) \in \mathbb{N}^n$, let $L(\mathbf{m})$ denote the set of permutations of the multiset $M = \{1^{m_1}, \ldots, n^{m_n}\}$, with the usual multiset notation where exponents represent multiplicity.
An order relation on $L(\mathbf{m})$ can be succinctly defined via the following cover relations. For $s,t \in L(\mathbf{m})$, we say that $t$ covers $s$ if and only if $s$ and $t$ differ only in swapping an adjacent pair of entries, which are in numerical order in $s$ but are reversed in $t$. For instance, we have that $(1~1~2~3~2) \lessdot (1~2~1~3~2)$ in $L(2,2,1)$. The poset $L(\mathbf{m})$ is called the \emph{multinomial Newman lattice} and was originally introduced by Bennett and Birkhoff in \cite{Birkhoff94}.

The multinomial Newman order can also be defined explicitly as follows.
Consider the \emph{set}, $$S=\{1_1, \ldots, 1_{m_1}, \ldots, n_1, \ldots, n_{m_n}\},$$ which we endow with the lexicographic total ordering $1_1 \lessdot 1_2 \lessdot \ldots \lessdot 1_{m_1} \lessdot \ldots \lessdot n_{m_n}$.
Note that we may identify a multipermutation $s \in L(\mathbf{m})$ with a unique corresponding permutation $\pi \in \mathfrak{S}_S$ where we require that copies of the same elements appear in lexicographic order, that is to say $i_j$ appears before $i_k$ in $\pi$ for all $i\in [n]$ and all $j < k$. For example, the multipermutation $(1~2~1~3~2) \in L(2,2,1)$ is identified with the permutation $(1_1 ~2_1~ 1_2~ 3_1~ 2_2)$. Let $\iota : L(\mathbf{m}) \to \mathfrak{S}_S$ denote this mapping.

One can show that $\iota$ is in fact an order-isomorphism from $L(\mathbf{m})$ to a principal ideal of the permuhatedron $\mathfrak{S}_S$;  it follows that $L(\mathbf{m})$ is also a lattice \cite{Birkhoff94, gen_permutahedra, Santocanale2007}. Specifically, the multinomial Newman lattice is isomorphic to the principal ideal generated by $(n_{1} \ldots n_{m_n} \ldots 1_1 \ldots 1_{m_1})$. Thus $L(\mathbf{m})$ has (necessarily unique) minimal and maximal elements, denoted $\hat{0}$ and $\hat{1}$ respectively. $\hat{0}$ is the identity permutation $(1~1\ldots 1~2~2 \ldots n)$ while $\hat{1}$ is the ``fully reversed" permutation $(n~n \ldots n ~(n-1)~(n-1) \ldots 1)$.


\section{The Unlabeled Barcode Poset}\label{reps}

\subsection{Barcodes as Multiset Permutations}\label{perm}

Let $L(2^n)$ denote the multinomial Newman lattice $L(\mathbf{m})$ for $\mathbf{m} = (2, 2, \ldots, 2) \in \mathbb{N}^n$.
Consider a strict barcode $B = \{(b_i, d_i)\}_{i=1}^n$. By definition, the birth/death times in $B$ are distinct, so the set $T = \{b_1, d_1, \ldots, b_n, d_n\}$ can be linearly ordered: $t_{i_1} < t_{i_2} < \ldots < t_{i_{2n}}$.
The indices in this order induce a multipermutation $(i_1, i_2, \ldots, i_{2n}) \in L(2^n)$.
For example, if $B_1$ is the strict barcode with 3 bars given by $b_1 = 1.0,\, d_1 = 2.0,\, b_2 = 1.5,\, d_2 = 3.0,\, b_3 = 2.5,\, d_3 = 2.75$, then the birth/death times in $T$ are ordered: $b_1 < b_2 < d_1 < b_3 < d_3 < b_2$. This produces the permutation $(1~2~1~3~3~2) \in L(2^3)$.

Let $f : \mathcal{B}^n_{st} \to L(2^n)$ denote this mapping, then we can use $f$ to define an equivalence relation on $\mathcal{B}^n_{st}$, which we call \emph{labeled barcode equivalence}. We leave it to the reader to verify that the following definition satisfies the necessary requirements of an equivalence relation.

\begin{definition}[Labeled barcode equivalence]
Let $B = \{(b_i, d_i)\}_{i=1}^n$ be a strict barcode whose birth/death times, $T = \{b_1, d_1, \ldots, b_n, d_n\}$, are ordered: $t_{i_1} < t_{i_2} < \ldots < t_{i_{2n}}$. Define the map $f : \mathcal{B}^n_{st} \to L(2^n)$ so that $f(B) = (i_1~ i_2~ \ldots~ i_{2n})$. Then, we say two barcodes $B_1, B_2 \in \mathcal{B}^n_{st}$ are \emph{labeled barcode equivalent} if and only if $f(B_1) =f(B_2)$.
\end{definition}

Perhaps unsurprisingly, one issue with labeled barcode equivalence is its dependence on a given labeling. For example, consider the strict barcode $B_2$ given by: $b_2 = 1.0,\, d_2 = 2.0,\, b_1 = 1.5,\, d_1 = 3.0,\, b_3 = 2.5,\, d_3 = 2.75$. Clearly, $B_2$ is the same barcode as $B_1$ from the prior example, except that the labels of bars 1 and 2 have been swapped. As a result, $f(B_1) = (1~2~1~3~3~2)$ while $f(B_2)= (2~1~2~3~3~1)$ and hence the two are not labeled barcode equivalent. Ideally, we would want any two barcodes to remain equivalent if their images under $f$ are the same up to some relabeling.

To that end, consider the group action $\Sigma : \mathfrak{S}_n \times L(2^n) \to L(2^n)$ given by $\Sigma(\pi, s) =\pi \circ s$, where
$\pi \circ s$ denotes $\pi$ acting element-wise on the terms in $s$. Let $[s]$ denote the orbit of $s\in L(2^n)$ and let $L(2^n)/\mathfrak{S}_n$ denote the set of all orbits, which we will call the space of \emph{combinatorial barcodes} with $n$ bars.
We can use this action to define a label-agnostic equivalence relation on the space of barcodes, which we call \emph{combinatorial barcode equivalence}. Again, we ask the reader to verify the following definition satisfies the necessary requirements.


\begin{definition}[combinatorial barcode equivalence]
Let $B_1, B_2 \in \mathcal{B}_n^{st}$ and let $f : \mathcal{B}^n_{st} \to L(2^n)$ denote the labeled barcode equivalence map. Let $g: \mathcal{B}^n_{st} \to L(2^n)/\mathfrak{S}_n$ denote the map given by $g(B) = [f(B)]$.
We say $B_1, B_2$ are \emph{combinatorially equivalent} if and only if $g(B_1) = g(B_2)$. 
\end{definition}




\subsection{Ordering Combinatorial Barcodes}\label{poset}

We must now the decide on a convention for choosing a representative of a combinatorial barcode, which again is an orbit in $L(2^n)$. Let $s\in L(2^n)$ and consider the substring $\tau_s$ of $s$ given by the first occurrence of each element (birth times). Now consider $\tau_s$ as a permutation in $\mathfrak{S}_n$ given in one-line notation. Then, the action of $\tau_s^{-1}$ on $s$ relabels $s$ so that the birth times now appear in order, i.e., the first copy of 1 appears before the first copy of 2, which appears before the first copy of 3, and so on.
For example, for $s = (2~1~4~1~3~3~2~4)\in L(2^n)$ we have that $\tau_s = (2~1~4~3)$ and so $\tau_s^{-1}\circ s = (1~2~3~2~4~4~1~3)$.
Note that if $u,v \in [s]$ we have
$\tau_u^{-1}\circ u = \tau_v^{-1}\circ v$. Hence the map $\psi: L(2^n)/\mathfrak{S}_n \to L(2^n)$ given by $\psi([s]) = \tau_s^{-1}\circ s$ does not depend on the choice of $s$. Therefore, we define the \emph{canonical representative} of an orbit, $[s]$ to be the permutation $\psi(s) \in L(2^n)$.


Observe that $\psi$ is injective, since by definition $\tau_s^{-1} \circ s \in [s]$. Therefore, $\psi$
gives us an embedding of the combinatorial barcodes back into the multinomial Newman lattice, which we can use to extend the multinomial Newman order to $L(2^n)/\mathfrak{S}_n$.


\begin{definition}
Let $[s],[t] \in L(2^n)/ \mathfrak{S}_n$ and let $\psi: L(2^n)/\mathfrak{S}_n \to L(2^n)$ be the map which sends each orbit to its canonical representative. Then, let $\leq_c$ be denote relation given by $[s] \leq_c [t]$ if and only if $\psi([s]) \leq \psi([t])$, where $\leq$ denotes the multinomial Newman order. We call the pair $(L(2^n)/\mathfrak{S}_n, \leq_c)$ the \emph{combinatorial barcode poset} or \emph{combinatorial barcode lattice} (we will see in a moment that this name is justified).
\end{definition}

\begin{remark}
Note that we could have directly defined a map between $\mathcal B^n_{st}$ and $L(2^n)$ that sends each barcode $B$ to the invariant $\psi(g(B))$. In fact, this is the approach that the authors in \cite{kanari2020trees} and \cite{ bruck2021stratifying} took when defining their own maps. While our construction is less direct, we chose it because it emphasizes the importance of the action of the symmetric group.
\end{remark}

Now, recall that the multinomial Newman lattice itself can be defined using the embedding $\iota: L(\mathbf{m}) \hookrightarrow \mathfrak{S}_S$. Hence, we have that,
\begin{equation}
    [s] \leq_c [t] 
    \iff \inv(\iota \circ \psi([s]))\subseteq \inv(\iota \circ \psi([t])).
\end{equation}
The roles of $f,g, \psi, \iota$ and $\Sigma$ are summarized in the diagram (\ref{cd}). By abuse of notation, we let $\Sigma$ also denote the map that sends each element to its orbit. Note, this diagram is \emph{not} commutative as $\psi \circ g \neq f$, although it is the case that $\Sigma \circ f = g$.


\begin{equation}\label{cd}
\begin{tikzcd}
                                                           &  & \mathcal{B}_{st}^n \arrow[dd, "f"] \arrow[lldd, "g"', bend right] &  &                                        \\
                                                           &  &                                                                   &  &                                        \\
L(2^n)/\mathfrak{S}_n \arrow[rr, "\psi", hook, bend right] &  & L(2^n) \arrow[ll, "\Sigma"', bend right] \arrow[rr, "\iota", hook]     &  & \mathfrak{S}_s \cong \mathfrak{S}_{2n}
\end{tikzcd}
\end{equation}

A remarkable property of the barcode poset is that it admits an elegant, alternate construction based on inversions that does not require first ``translating" to the symmetric group.


\begin{definition}
Let $[s]$ be a combinatorial barcode and without loss of generality let $s$ denote its canonical representative, $\psi([s])$. Then, the \emph{inversion multiset} of $[s]$ is the multiset of pairs $\{(j,i)^{a_{ij}}: 1\leq i < j \leq n\}$ where $a_{ij}$ is equal to the number of pairs of indices $(k,\ell)$ such that  $s_k = i, s_\ell =j$ and $k >\ell$. We will denote the inversion multiset of $[s]$ by $\invm([s])$ or by $\invm(s)$ when no confusion can occur.
\end{definition}

Stated simply, the inversion multiset has as elements the pairs $(j,i)$, $i<j$, with multiplicity equal to the number of pairs of $i$'s and $j$'s that appear out of order in $s$.
For example, $\invm((1~2~3~2~4~4~1~3)) = \{(2,1)^2, (3,1)^1, (4,1)^2, (3,2)^1, (4,3)^1 \}$.
Now, for $[s],[t], \in L(2^n)/\mathfrak{S}_n$, write $[s] \prec [t]$, if $\invm(s)\subseteq \invm(t)$; recall, given multisets $A= \{x_1^{a_1}, \ldots, x_n^{a_n}\}$, $B=\{x_1^{b_1}, \ldots, x_n^{b_n}\}$ we say that $A \subseteq B$ if $a_i \leq b_i$ for all $i \in [n]$.

\begin{prop}\label{construction}
 For $[s],[t] \in L(2^n)/\mathfrak{S}_n$, we have that $[s] \prec[t] \iff [s] \leq_c [t]$.
\end{prop}
\begin{proof}
Let $[s] \in L(2^n)/\mathfrak{S}_n$ and without loss of generality let $s = \psi([s])$. Let $(j,i)$ be an element in $\invm(s)$ with multiplicity $k$. Recall that in the canonical representative, $s$, the first copy of $i$ appears before the first copy of $j$ for all $i <j$. Therefore, the copies of $i,j$ must appear according to one of three patterns:
$$(1) \, i \ldots i \ldots j \ldots j, \,(2)\, i \ldots j \ldots i \ldots j, \,(3)\, i \ldots j \ldots j \ldots i\,.$$
It follows that $k\in \{0,1,2\}$, specifically,  $k=0$ when $s$ contains pattern (1), $k=1$ when $s$ contains pattern (2), and $k=2$ when $s$ contains pattern (3).

Now, let $A_{ij} = \{(j_1, i_1), (j_2, i_1), (j_1, i_2), (j_2,i_2)\}$ be the set of all possible inversions involving $i$ and $j$ in $\iota(s)$, which is a permutation of the set $S=\{1_1,1_2, 2_1, 2_2, \ldots, n_1, n_2\}$. It follows that:
\begin{equation*}
\inv(\iota(s)) \cap A_{ij} =
    \begin{cases}
      \emptyset &,\, k = 0 \\
      \{(j_1, i_2) \}&, \, k=1 \\
      \{(j_1, i_2), (j_2,i_2)\} &, \, k=2
   \end{cases}.
\end{equation*}
Now suppose $[s] \prec [t]$. Then, $(j,i) \in \invm(t)$ with multiplicity $\ell$ and necessarily $\ell \geq k$. As a result, $(\inv(\iota(s)) \cap A_{ij}) \subseteq (\inv(\iota(t)) \cap A_{ij})$. Applying this argument to all pairs in $\invm(s)$, it follows that $\inv(\iota(s)) \subseteq \inv(\iota(t))$ and thus $\iota(s) \leq_W \iota(s) \implies [s]\leq_c [t]$.

Moreover, this argument is reversible in that we can deduce the multiplicity of $(j,i) \in \invm(s)$ from $\inv(\iota(s)) \cap A_{ij}$. Hence, we also have that $[s]\leq_c [t]\implies [s] \prec [t]$, as desired.
\end{proof}

\begin{remark}
We emphasize that the construction above does not work for the multinomial Newman lattice $L(2^n)$. For example, $(1~2~2~1)$ and $(2~1~1~2)$ both have the inversion multiset $\{(2,1)^2\}$ but their inversion sets are $\{(2_1,1_2), (2_2,1_2) \}$ and $\{(2_1, 1_1), (2_1, 1_2)\}$, respectively. Hence we do not have the necessary one-to-one correspondence between inversion multisets of elements in $L(2^n)$ and their inversions sets after embedding in $\mathfrak{S}_S$.
\end{remark}

Now, let $B = \{(b_i,d_i)\}_{i=1}^n$ be a barcode. The proof of Proposition \ref{construction} illuminates a simple way of computing the rank of the combinatorial barcode $g(B) \in L(2^n)/\mathfrak{S}_n$. For each pair of bars $i,j$ with $b_i < b_j$, define the \emph{crossing number} of bars $i$ and $j$ to be:
\begin{equation}
\cross(i,j)=
    \begin{cases}
    \begin{aligned}
      0 &,\,\, d_i < b_j  &&\text{ (disjoint)}\\
      1 &, \,\, b_j < d_i < d_j &&\text{ (stepped)}\\
      2 &, \,\, d_j < d_i &&\text{ (nested)}
      \end{aligned}
   \end{cases}.
\end{equation}
If $B$ is given as a barcode diagram then the crossing number of a pair of bars is easy to read: it equals 0 when the bars are disjoint, 1 when they are stepped, and 2 when they are nested (see Figure \ref{bars}).

\begin{figure}[h]
     \centering
     \begin{subfigure}[b]{0.3\textwidth}
         \centering
         \includegraphics[width=\textwidth]{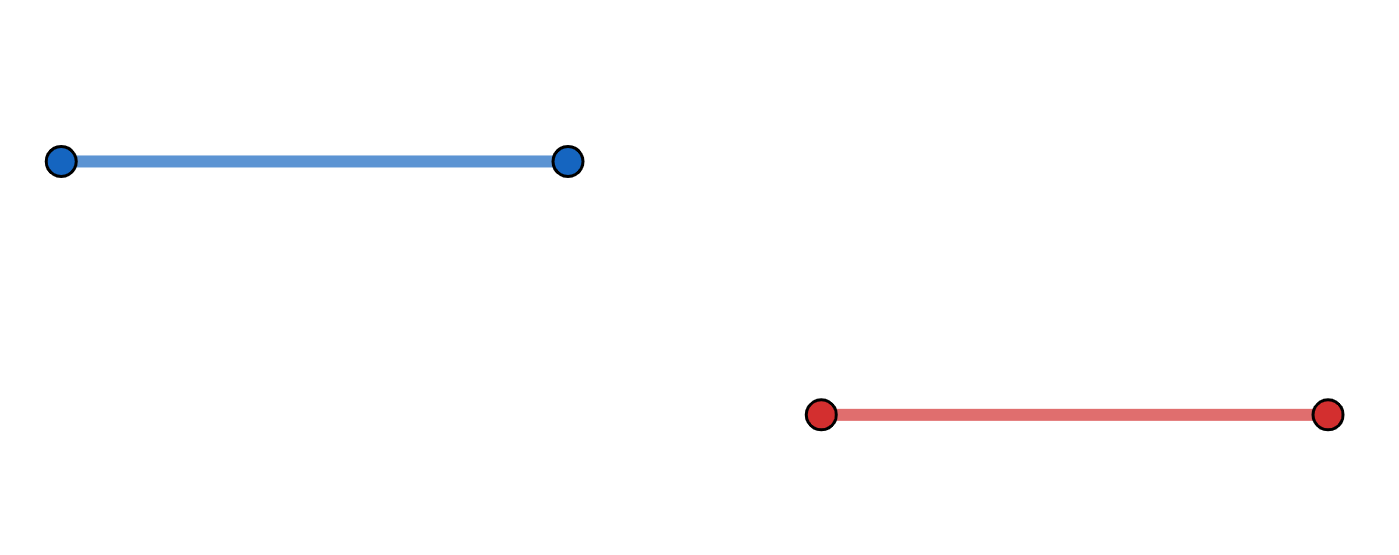}
         \caption{Disjoint.}
         \label{disjoint}
     \end{subfigure}
     \hfill
     \begin{subfigure}[b]{0.3\textwidth}
         \centering
         \includegraphics[width=\textwidth]{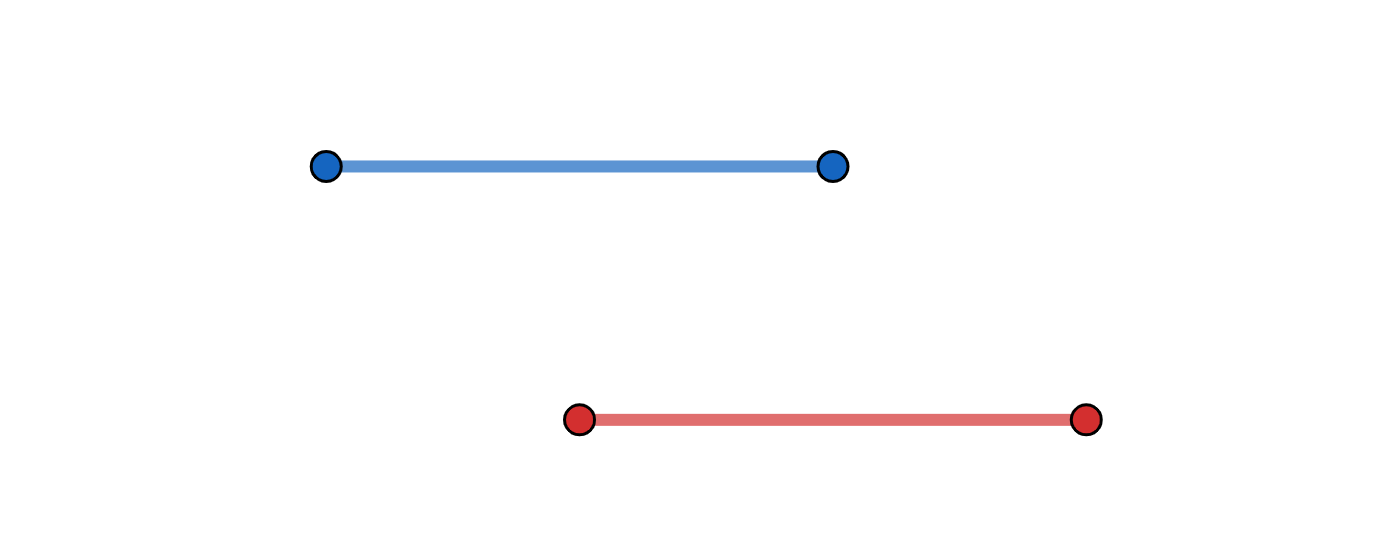}
         \caption{Stepped.}
         \label{stepped}
     \end{subfigure}
     \hfill
     \begin{subfigure}[b]{0.3\textwidth}
         \centering
         \includegraphics[width=\textwidth]{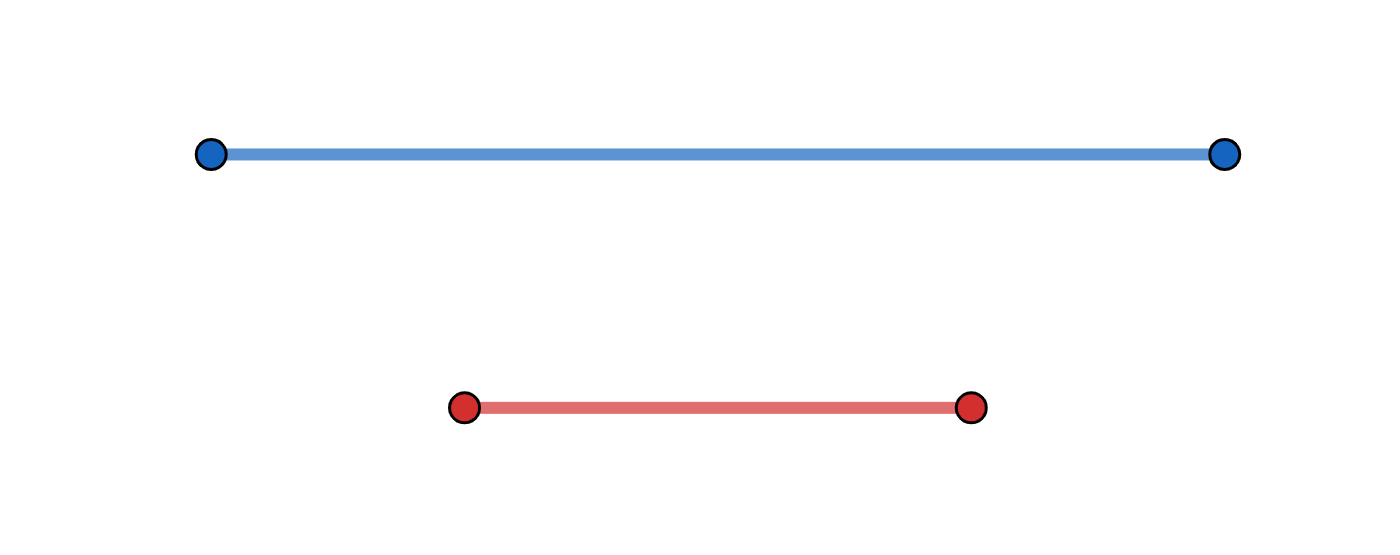}
         \caption{Nested.}
         \label{nested}
     \end{subfigure}
        \caption{The three possible arrangements of a pair of bars.}
        \label{bars}
\end{figure}
\begin{cor}\label{crossing_num_cor}
Let $B = \{(b_i,d_i)\}_{i=1}^n$ be a barcode. Then,
$$ \rho(g(B)) = \sum_{b_i < b_j} \cross(i,j).
$$
\end{cor}
\begin{proof}
Assume without loss of generality that $B$ has the same labeling as the canonical representative of $g(B)$. Let $s$ denote this canonical representative. Then $\cross(i,j)$ is exactly the multiplicity of $(j,i)$ in $\invm(s)$. Finally, note that $\lvert \invm(s)\rvert = \vert\inv(\iota(s))\rvert = \rho(g(B))$, where elements are counted with multiplicity in the inversion multiset.
\end{proof}

\begin{figure}
    \centering
    \includegraphics[width=0.95\textwidth]{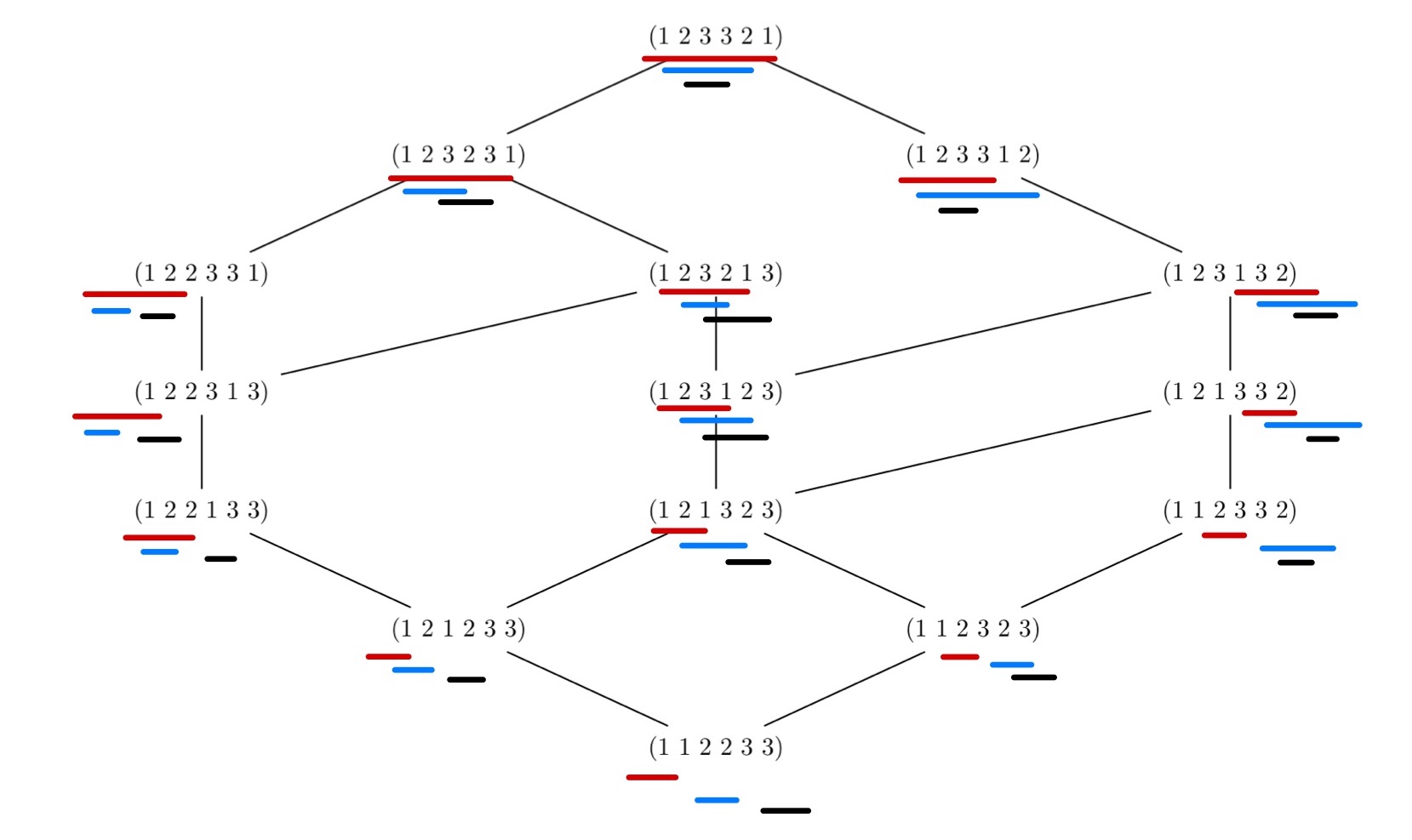}
\caption{Hasse diagram of $L(2^3)/\mathfrak{S}_3$. Below each element, $[s]$, is a barcode diagram for which $g(B) =[s]$, illustrating Corollary \ref{crossing_num_cor}.}
\label{hasse_barcode}
\end{figure}

Our first main result is that the barcode poset, $L(2^n)/\mathfrak{S}_n$, is a isomorphic to principal ideal of the multinomial Newman lattice $L(2^n)$ and hence is a lattice.

\begin{theorem}\label{lattice_thm}
The barcode poset $(L(2^n)/\mathfrak{S}_n, \leq_u)$ is isomorphic to the principal ideal of the multinomial Newman lattice, $L(2^n)$ generated by the ``fully nested" permutation: $(1~2~\ldots~(n-1)~n~n~(n-1)~\ldots~2~1)$. Consequently, $(L(2^n)/\mathfrak{S}_n, \leq_u)$ is a lattice.
\end{theorem}
\begin{proof}
From the multiset definition, it is clear that $(L(2^n)/\mathfrak{S}_n, \leq_u)$ contains a maximal element $\hat{1}$, with canonical representative $\psi(\hat{1}) = (1~2~\ldots~~n~n~\ldots~2~1)$.
By abuse of notation let $\hat{1}$ also denote this canonical representative.
We wish to show that the map $\psi$ is an isomorphism onto the principal ideal generated by $\hat{1}$, which we denote $I(\hat{1})$.
We have already established that $\psi$ is injective, so we need only establish that $\image(\psi) = I(\hat{1})$. Since $\psi$ is order preserving and $\hat{1}$ is maximal, it follows that $\psi(s) \leq \psi(\hat{1})$ for all $s\in (L(2^n)/\mathfrak{S}_n, \leq_u)$. Hence $\image(\psi) \subseteq I(\hat{1})$.

Now, let $s \in I(\hat{1})$. Let $\tau_s\in \mathfrak{S}_n$ be the permutation given by the string of birth times in $s$; recall that $\psi$ is defined as $\psi([s]) = \tau_s^{-1} \circ s$. If $\tau_s$ is not the identity permutation, then it follows that there exists a pair $i<j$ for which the first copy of $j$ appears before the first copy of $i$ in $s$. Hence, $(j_1,i_1)\in \inv(\iota(s))$. However, $\psi(s)\leq \psi(\hat{1})$ implies that $\inv(\iota(s)) \subseteq \inv(\iota(\hat{1}))$ and $(k_1, \ell_1) \notin \inv(\iota(\hat{1}))$ for any $k>\ell$. Hence, we have a contradiction. Therefore it must be the case that $\tau_s  =\Id_n$.
Finally, note that $\psi([s]) = \tau_s^{-1}\circ s =s$, so $s\in \image(\psi)$. Thus, $I(\hat{1}) \subseteq \image(\psi)$, completing the proof.
\end{proof}

\section{Power $k$ Barcode Lattice}\label{big_k}
In this section we generalize the construction in Section \ref{reps}, producing an entire family of multipermutations associated to barcodes. We will show that these new multipermutations share the same lattice structure as before, while also providing increasingly detailed information about the arrangement of the bars in a barcode. Ultimately, we show that for a large class of barcodes, these discrete invariants can provide bounds on two classical, continuous metrics on barcodes: the Wasserstein and bottleneck distances.

We first recall for the reader a mapping, defined in \cite{kanari2020trees, tressII}, from the space of strict barcodes with $n$ bars to the symmetric group $\mathfrak{S}_n$. Let $B\in \mathcal{B}_{st}^n$; by ordering the death times increasingly such that $d_{i_1} < d_{i_2} < \ldots <d_{i_n}$, the indexing set $[n]$ gives a permutation $\sigma_d \in \mathfrak{S}_n$ defined by $\sigma_B(k) = i_k$, i.e., $\sigma_B$ is the unique permutation such that $d_{\sigma_B(1)} < d_{\sigma_B(2)} < \ldots <d_{\sigma_B(n)}$. Ordering the birth times gives another permutation $\tau_B$, which is exactly $\tau_{f(B)}$ under the notation we used to define the canonical representative map $\psi$ in Section \ref{poset}. Thus, we have a map $\phi: \mathcal{B}_{st}^n \to \mathfrak{S}_n$ given by $\phi(B) = \tau_B^{-1} \cdot \sigma_B$ which tracks the ordering of the death values with respect to the birth values.

For example, if $B_1$ is the strict barcode with 3 bars given by $b_2 = 1.0,\, d_2 = 2.0,\, b_1 = 1.5,\, d_1 = 3.0,\, b_3 = 2.5,\, d_3 = 2.75$, then the birth/death times in $T$ are ordered: $b_2 < b_1 < d_2 < b_3 < d_3 < b_1$. So $\tau_B = (2~1~3)$, $\sigma_B = (2~3~1)$ and $\phi(B) = (1~3~2)$.

The map $\psi(g(B))$ provides an alternate invariant to $\phi(B)$, in this case a multipermutation, which in fact contains $\phi(B)$ as a sub-permutation; $\phi(B)$ is exactly the permutation you get by looking at the second occurrence of each element in $\psi(g(B))$.

\begin{prop}\label{kanari_containment}
Let $B\in \mathcal{B}_{st}^n$ and let $\phi$ denote the map from $\mathcal{B}_{st}^n \to \mathfrak{S}_n$ defined in \cite{kanari2020trees}. Then $\phi(B)$ is a sub-permutation of the multipermutation $\psi(g(B))$.
\end{prop}

It follows that the invariant $\psi(g(B))$ is more sensitive than the invariant $\phi(B)$ because it captures the relative positions of the birth times along with the death times. This begs the question: \emph{is there a further generalization of this construction where we consider more points in each bar rather than just the birth/death times?}

To that end, we must first determine a sensible way of selecting more points from each bar. A natural choice is to take the endpoints of all the intervals we get when splitting each bar into $2^k$ sub-intervals of equal length, where $k \in \mathbb{Z}_{\geq 0}$. For instance,  when $k=0$ this gives just the endpoints of each bar, which produces the barcode lattice, and when $k=1$ this gives us the endpoints and midpoint of each bar. For general $k$, we get the $(2^k+1)$-many points $\big\{b_i + \ell\frac{d_i-b_i}{2^k}: \ell= 0,\ldots, 2^k\big\}$ for each bar $(b_i, d_i)$.
We consider this choice natural because the points given by higher values of $k$ contain all points given by smaller values.

One problem with this method is that even strict barcodes might produce repeated points when $k>0$. For example, if $B = \{(-1,1), (-2,2)\}$, then although all birth/death times are distinct, the two bars both have midpoint $0$. In order to form combinatorial invariants analogous to the $k=0$ case, we require that situations like this do not occur. For that reason, we define the notion of \emph{$k$-strict barcodes}.

\begin{definition}
A barcode $B = \{(b_i,d_i)\}_{i=1}^n$ is called \emph{$k$-strict} if $$\big\{b_i + \ell\frac{d_i-b_i}{2^k}: \ell= 0,\ldots, 2^k\big\} \bigcap \big\{b_j + \ell\frac{d_j-b_j}{2^k}:\ell= 0,\ldots, 2^k\big\} = \emptyset$$ for all $i\neq j$. We denote the set of $k$ strict barcodes with $n$ bars by $\mathcal{B}_k^n$ and note that $\mathcal{B}_{0}^n = \mathcal{B}_{st}^n$. Furthermore, we let $\mathcal{B}_{\infty}^n = \bigcap_{k=0}^\infty \mathcal{B}_k^n$, and call this the space of \emph{fundamentally strict} barcodes with $n$ bars.
\end{definition}

\begin{remark}
Although the definition of fundamental strictness seems quite severe, in practice barcodes arising from statistical processes like TDA will often be fundamentally strict when the initial data comes from some continuous distribution.
\end{remark}
We may now define analogues of the maps $f$ and $g$ from Section \ref{perm} for higher values of $k$.

Let $k\in \mathbb{Z}_{\geq0}$ and let $B = \{(b_i,d_i)\}_{i=1}^n$ be a $k$-strict barcode with $n$ bars. Taking the union of the points $\big\{b_i + \ell\frac{d_i-b_i}{2^k}: \ell= 0,\ldots, 2^k\big\}$ for each $i \in [n]$ produces $n(2^{k}+1)$ points: $t_{i_1} <t_{i_2} < \ldots < t_{i_{n(2^k+1)}}$. Thus, we can define the map $f: \mathcal{B}_k^n \to L((2^k+1)^n)$, where $L((2^k+1)^n) = L((2^k+1), \ldots, (2^k+1))$, such that $$f_k(B) = (i_1~i_2~\ldots~i_{n(2^k+1)}).$$

As before, let $\mathfrak{S}_{n}$ act element-wise on the multipermutations in $L((2^k+1)^n)$. Then define the map $g_k : \mathcal{B}_k^n \to L((2k+1)^n) / \mathfrak{S}_{n}$ such that $g_k(B) = [f(B)]$, where again $[s]$ denotes the orbit of $s$ under $\mathfrak{S}_{n}$. Finally, define the canonical representative map $\psi_s: L((2^k+1)^n) / \mathfrak{S}_{n} \to L((^2k+1)^n)$ such that $\psi_k([s]) = \tau_s^{-1} \circ s$, where, as before, $\tau_s$ is the permutation in $\mathfrak{S}_{n}$ resulting from the ordering of the birth times of each bar in $s$. By the same argument as before, the map $\psi_k$ is injective and does not depend on the choice of representative $s$.

Thus, $\psi_k$ defines an embedding of $L((2^k+1)^n) / \mathfrak{S}_{n}$ into $L((2^k+1)^n)$, which induces a poset structure on $L((2^k+1)^n) / \mathfrak{S}_{n}$, which we call the \emph{power $k$ barcode poset}.

\begin{definition}
Let $[s],[t] \in L((2^k+1)^n)/ \mathfrak{S}_n$ and let $\psi_k: L((2^k+1)^n)/\mathfrak{S}_n \to L((2^k+1)^n)$ be the map which sends each orbit to its canonical representative. Then, let $\leq_k$ be denote relation given by $[s] \leq_k [t]$ if and only if $\psi_k([s]) \leq \psi_k([t])$, where $\leq$ denotes the multinomial Newman order. We call the pair $(L((2^k+1)^n)/\mathfrak{S}_n, \leq_k)$ the \emph{power $k$ barcode poset} or \emph{power $k$ barcode lattice} (we will see in a moment that this name is justified). Moreover, if $B$ is a $k$-strict barcode we say that $\psi_k(g_k(B))$ is the \emph{power $k$ invariant of $B$}.
\end{definition}
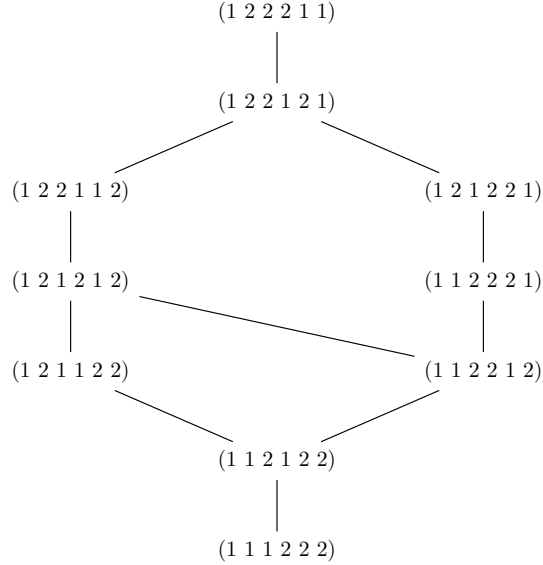
\begin{figure}[h]
    \centering
\begin{tikzcd}[scale cd = 0.75]
                                  & (1~2~2~2~1~1)                                         &                                                       \\
                                  & (1~2~2~1~2~1) \arrow[u, no head]                      &                                                       \\
(1~2~2~1~1~2) \arrow[ru, no head] &                                                       & (1~2~1~2~2~1) \arrow[lu, no head]                     \\
(1~2~1~2~1~2) \arrow[u, no head]  &                                                       & (1~1~2~2~2~1) \arrow[u, no head]                      \\
(1~2~1~1~2~2) \arrow[u, no head]  &                                                       & (1~1~2~2~1~2) \arrow[u, no head] \arrow[llu, no head] \\
                                  & (1~1~2~1~2~2) \arrow[ru, no head] \arrow[lu, no head] &                                                       \\
                                  & (1~1~1~2~2~2) \arrow[u, no head]                      &
\end{tikzcd}
    \caption{Hasse diagram of $L(3^2)/\mathfrak{S}_2$.}
    \label{multi_hasse}
\end{figure}

Recall that there exists a map $\iota$, defined earlier, which embeds $L(2^n)$ into $\mathfrak{S}_{2n}$. An analogous map, which by abuse of notation we will also refer to as $\iota$, embeds $L((2^k+1))^n$ into $\mathfrak{S}_{n(2^k+1)}$ by identifying $\mathfrak{S}_{n(2^k+1)}$ with permutations of the set $S = \{1_1,1_2, \ldots, 1_{2^k+1}, \ldots, n_{2^k+1}\}$.
Hence, as before, we have that,
\begin{equation}
    [s] \leq_k [t] 
    \iff \inv(\iota \circ \psi_k([s]))\subseteq \inv(\iota \circ \psi_k([t])).
\end{equation}

As in the $k=0$ case, we claim that the power $k$ barcode poset is in fact a lattice.
We note that the proof of Theorem \ref{multi-k} below is a direct generalization of the arguments in the proof of Theorem \ref{lattice_thm}.

\begin{theorem}\label{multi-k}
The power $k$ barcode poset $(L((2^k+1)^n)/\mathfrak{S}_n, \leq_k)$ is isomorphic to a principal ideal of the multinomial Newman lattice, $L((2^k+1)^n)$. Consequently, $(L((2^k+1)^n)/\mathfrak{S}_n, \leq_k)$  is a lattice.
\end{theorem}
\begin{proof}
As in the $k=0$, it is clear that $(L((2^k+1)^n)/\mathfrak{S}_n, \leq_k)$ contains a maximal element $\hat{1}$. In general, the canonical representative of $\hat{1}$ is the multipermutation
$\psi_k(\hat{1})$ whose one-line notation is the integers $1$ to $n$, followed by the remaining $2^k$ copies of $n$, followed by the remaining $2^k$ copies of $(n-1)$, and so on, terminating with the remaining $2^k$ copies of 1. For example, when $k=1$ and $n=3$, we have $3$ copies of each integer and $\psi_k(\hat{1}) = (1~2~3~3~3~2~2~1~1)$. It is clear that this element is maximal, as every pair of elements is inverted except the first copy of each integer which we require appear in order.

By abuse of notation, let $\hat{1}$ denote its canonical representative $\psi_k(\hat{1})$.
We claim that the map $\psi_k$ is an isomorphism onto the principal ideal generated by $\hat{1}$, which we denote $I(\hat{1})$.

We have already established that $\psi_k$ is injective, so we need only establish that $\image(\psi_k) = I(\hat{1})$. Since $\psi_k$ is order preserving and $\hat{1}$ is maximal, it follows that $\psi_k(s) \leq \psi_k(\hat{1})$ for all $s\in (L(2^n)/\mathfrak{S}_n, \leq_k)$. Hence $\image(\psi_k) \subseteq I(\hat{1})$.

Now, let $s \in I(\hat{1})$. Let $\tau_s\in \mathfrak{S}_n$ be the permutation given by the string of birth times in $s$; recall that $\psi_k$ is defined as $\psi_k([s]) = \tau_s^{-1} \circ s$. If $\tau_s$ is not the identity permutation, then it follows that there exists a pair $i<j$ for which the first copy of $j$ appears before the first copy of $i$ in $s$. Hence, $(j_1,i_1)\in \inv(\iota(s))$. However, $\psi_k(s)\leq \psi_k(\hat{1})$ implies that $\inv(\iota(s)) \subseteq \inv(\iota(\hat{1}))$ and $(k_1, \ell_1) \notin \inv(\iota(\hat{1}))$ for any $k>\ell$. Hence, we have a contradiction. Therefore it must be the case that $\tau_s  =\Id_n$.

Finally, note that $\psi_k([s]) = \tau_s^{-1}\circ s =s$, so $s\in \image(\psi_k)$. Thus, $I(\hat{1}) \subseteq \image(\psi_k)$, completing the proof.
\end{proof}

\subsection{Increasing Descriptive Power of $g_k(B)$}\label{geom}

In the same way that the invariant $\phi(B)$ is a sub-permutation of $\psi(g(B))$, the invariant $\psi_j(g_j(B))$ is a sub-permutation of $\psi_k(g_k(B))$ for all $k > j$. Specifically, note that if we delete every other occurrence (beginning with the second) of $i$ in $f_{k+1}(B)$, for each $i\in[n]$, then the resulting multipermutation is precisely $f_k(B)$. Hence, we have a map $\delta_k: (L((2^{k+1}+1)^n) \to (L((2^k+1)^n)$ such that $\delta_k \circ f_{k+1} = f_k$.

For instance,  consider the barcode $B$ given by $b_1 = 1.0,\, d_1 = 2.5,\, b_2 = 1.5,\, d_2 = 4.0,\, b_3 = 3.0,\, d_3 = 3.5$. Taking $k=1$, we add the points $m_1 = 1.75,\, m_2 = 2.75,\,m_3 = 3.25$ so $f_1(B) = (1~2~1~1~2~3~3~3~2)$. The map $\delta_0$ deletes the points $m_i$, which gives $(\delta_0 \circ f_1)(B) = (1~2~1~3~3~2)= f_0(B)$.

Hence, we have the following lemma.
\begin{lemma}\label{inclusion_lem}
Let $B_1, B_2 \in \mathcal{B}_k^n$. If $g_k(B_1) = g_k(B_2)$, then $g_j(B_1) =g_j(B_2)$ for all $j <k$.
\end{lemma}

Thus, we see that increasing $k$ amounts to producing ever more sensitive invariants $g_k(B)$. These higher order invariants capture more nuanced information about the overlaps of pairs of bars. For instance, we have seen that if a barcode $B$ contains two nested bars then $g_0(B)$ will contain the pattern $(1~2~2~1)$. Going up a level, $g_1(B)$ confirms that the bars are nested but also tells us whether bar 2 is contained in the left half of bar 1, in the right half of bar 1, or whether it straddles the midpoint of 1 (see Figure \ref{nested_bars}).

In fact, we see that all the possible patterns of 1's and 2's tell you a distinct way the two bars intersect and how those intersections relate to the two halves of the bars.
By the same logic, higher values of $k$ provide even more granular intersection data, now relating to quartiles ($k=2$), octiles ($k=3$), etc., of the bars.


\begin{figure}[h]
     \centering
     \begin{subfigure}[b]{0.4\textwidth}
         \centering
         \includegraphics[width=\textwidth]{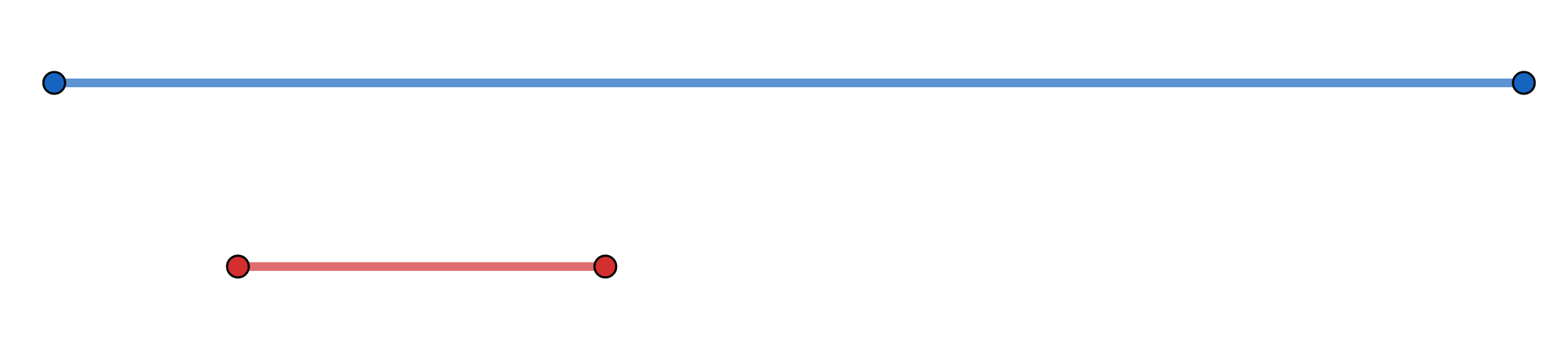}
         \caption{Left-nested bars, $g_1(B) = (1~2~2~2~1~1)$.}
         \label{left-nested}
     \end{subfigure}
     \hfill
     \begin{subfigure}[b]{0.4\textwidth}
         \centering
         \includegraphics[width=\textwidth]{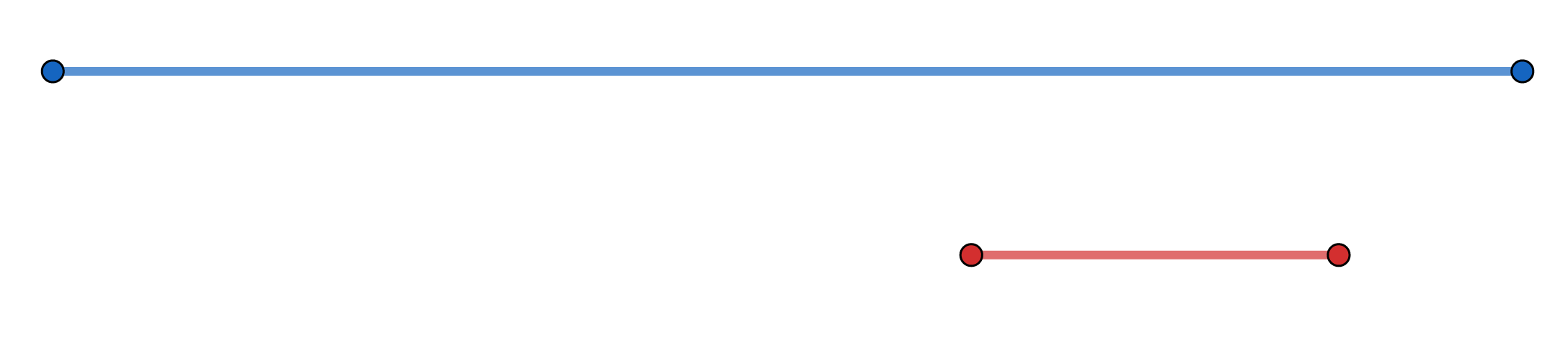}
         \caption{Right-nested bars, $g_1(B) = (1~1~2~2~2~1)$.}
         \label{right-nested}
     \end{subfigure}
        \caption{Two barcodes with the same power $0$ invariant, $g_k(B) = (1~2~2~1)$, but different power $1$ invariants.}
        \label{nested_bars}
\end{figure}


\subsection{Connection to Bottleneck and Wasserstein Distances}\label{distance}
From the observations above, it is clear that the power $k$ invariant of a barcode $B$ captures increasingly granular information about the relative positions of the bars as $k$ increases. In fact, we will show that these invariants are closely related to classical metrics between barcodes such as the bottleneck and Wasserstein distances for a special family of barcodes.

Before explaining further, we first recall for the reader the notion of an interval graph \cite{Lekkeikerker1962}.
\begin{definition}
Let $B = \{(b_i,d_i)\}_{i=1}^n$ be a barcode. The \emph{interval graph} of $B$ is the simple graph $G_B(V,E)$ where $V = [n]$ and an edge $(i,j)$ is in $E$ if and only if $(b_i,d_i) \cap(b_j,d_j) \neq \emptyset$.
\end{definition}

\begin{lemma}\label{graph_lem}
Let $B,B'$ be strict barcodes. If $g_k(B) = g_k(B')$ for some $k\geq 0$ then $G_B \cong G_{B'}$.
\end{lemma}
\begin{proof}
It is clear that the intersection of bars $i$ and $j$ can be determined from $\cross(i,j)$. Recall one can deduce $\cross\#(i,j)$ from the power 0 invariant of a barcode for all $i\neq j$. Therefore, the power 0 invariant, and by Proposition \ref{inclusion_lem} any power $k$ invariant, completely determines the interval graph of its associated barcode.
\end{proof}

\begin{theorem}\label{affine}
Let $B,B'$ be fundamentally strict barcodes with $n$ bars, where $B =\{(b_i,d_i)\}_{i=1}^n$ and $B' =\{(b_i',d_i')\}_{i=1}^n$. If $g_k(B) = g_k(B')$ for all $k\in \mathbb{N}$ and $G_B$ (equivalently $G_{B'}$) is connected, then there exist constants $\alpha>0$ and $\delta \in \mathbb{R}$ such that $B = \alpha B' +\delta$, where $ \alpha B' +\delta := \{(\alpha b_i'+ \delta,\alpha d_i' + \delta): i\in [n]\}$.
\end{theorem}
\begin{proof}
Without loss of generality assume that $B,B'$ are labeled according to increasing birth time, that is to say $b_1 <b_2 < \ldots <b_n$, and likewise for $B'$. Now, let $\alpha = \frac{d_1-b_1}{d_1'-b_1'}$, $\delta = b_1 -\alpha b_1'$, and define $T: \mathbb{R} \to \mathbb{R}$ be function $T(x) = \alpha x +\delta$. Observe that $T(b_1')  = b_1$, and that
\begin{align*}
    T(d_1') = \frac{d_1'(d_1-b_1)}{d_1'-b_1'} + b_1 - \frac{b_1'(d_1-b_1)}{d_1'-b_1'}
    = \frac{(d_1'-b_1')(d_1-b_1)}{d_1'-b_1'} +b_1
    = d_1.
\end{align*}
Now, let $i$ be a neighbor of 1 in $G_B$ (and hence in $G_{B'}$, by Lemma \ref{graph_lem}). We know at least one such $i$ exists since $G_B$ is connected. Take $k>0$ to be arbitrary and let $m$ denote the number of $1$'s that appear before the first $i$ in $\psi_k(g_k(B))$, or equivalently in $\psi_k(g_k(B'))$. Note $0 < m < 2^{k}$ since necessarily $b_1 < b_i < d_1$.
Thus, we have that
\begin{alignat*}{2}
    b_1 +(m-1)\frac{d_1-b_1}{2^k} &< \, b_i\, &&< b_1 +m\frac{(d_1-b_1)}{2^k}, \text{ and} \\
    b_1' +(m-1)\frac{d_1'-b_1'}{2^k} &<\, b_i'\, &&< b_1' +m\frac{(d_1'-b_1')}{2^k}.
\end{alignat*}
Since $\alpha >0$, it follows that
\begin{alignat*}{2}
    \alpha \Big(b_1' +(m-1)\frac{d_1'-b_1'}{2^k}\Big) +\delta &<\, T(b_i') \, &&< \alpha \Big(b_1' +m\frac{(d_1'-b_1')}{2^k}\Big) +\delta \\
    \implies b_1 +\frac{(m-1)(d_1-b_1)}{2^k} &<\, T(b_i') \,&&< b_1 +\frac{m(d_1-b_1)}{2^k}.
\end{alignat*}
Therefore, $\lvert b_i - T(b_i')\rvert < \frac{d_1-b_1}{2^k}$. Sending $k\to\infty$, it follows that $T(b_i') =b_i$.

Now, if $b_1 < d_i < d_1$, then repeating the argument above gives $T(d_i') = d_i$. Otherwise, we have that $b_1 < b_i < d_1 <d_i$. Let $\ell$ be the number of $i$'s that appear before the last 1 in $\psi_k(g_k(B))$. Note that $0 < \ell < 2^k+1$, then we have
\begin{alignat*}{2}
    \frac{\ell(d_i-b_i)}{2^k}
    &<\, d_1 -b_i \,&&<
    \frac{(\ell+1)(d_i-b_i)}{2^k}, \text{ and} \\
    \frac{\ell(d_i'-b_i')}{2^k}
    &<\, d_1' -b_i' \,&&<
    \frac{(\ell+1)(d_i'-b_i')}{2^k}.
\end{alignat*}
Again, as $\alpha >0$, it follows that,
\begin{alignat*}{2}
    \frac{\ell(\alpha d_i'- \alpha b_i')}{2^k}
    &<\, \alpha d_1' -\alpha b_i' \,&&<
    \frac{(\ell+1)(\alpha d_i'- \alpha b_i')}{2^k},
    \\
    \implies \frac{\ell(T(d_i')-T(b_i'))}{2^k}
    &<\, T(d_1') -T(b_i') \,&&<
    \frac{(\ell+1)(T(d_i')-T(b_i'))}{2^k},
    \\
    \implies
    \frac{\ell(T(d_i')-b_i)}{2^k}
    &<\, d_1-b_i\,&&<
    \frac{(\ell+1)(T(d_i')-b_i)}{2^k}.
\end{alignat*}
Rearranging the inequalities above gives,
\begin{alignat*}{2}
    \frac{\ell}{2^k}
    &<\, \frac{d_1 -b_i}{d_i-b_i} \,&&<
    \frac{(\ell+1)}{2^k}, \text{ and} \\
    \frac{\ell}{2^k}
    &<\, \frac{d_1 -b_i}{T(d_i')-b_i} \,&&<
    \frac{(\ell+1)}{2^k}.
\end{alignat*}
Sending $k\to \infty$ it follows that $\frac{d_1 -b_i}{d_i-b_i} = \frac{d_1 -b_i}{T(d_i')-b_i}$, from which we see that $T(d_i') =d_i$.
Hence, if $i$ is a neighbor of $1$, then $(T(b_i'), T(d_i')) = (b_i,d_i)$.
Finally, note that we can repeat the arguments above to show that this holds for all vertices connected to 1 by some finite path. Since $G_B$ is connected and finite, this gives the desired result.
\end{proof}

Thus, we see that the entire family of power $k$ invariants completely determine a barcode $B$ up to an affine transformation. Moreover, with an extra assumption on $B,B'$, we get bounds on the rate of convergence.
\begin{theorem}\label{convergence}
	Let $B,B'$ be $k$-strict barcodes with $n$ bars such that $g_k(B) = g_k(B')$. Suppose there exists a bar $(b_*,d_*) \in B$ (or equivalently in $B'$) which contains all others, that is to say $b_* \leq b_i$ and $d_* \geq d_i$ for all $i\in [n]$.
	Then there exist constants $\alpha > 0$ and $\delta \in \mathbb{R}$ such that
\begin{align*}
    d_\infty(B, \alpha B' +\delta) &\leq \frac{\lvert d_*-b_*\rvert}{2^k}, \text{ and} \\
d_q(B, \alpha B' +\delta) &\leq (n-1)^{\frac{1}{q}} \frac{\lvert d_*-b_*\rvert}{2^k}.
\end{align*}
\end{theorem}
\begin{proof}
	Without loss of generality assume that $B,B'$ are labeled according to increasing birth time, that is to say $b_1 <b_2 < \ldots <b_n$, and likewise for $B'$. Note that this implies that $(b_*,d_*) = (b_1,d_1)$. Moreover, since $\psi_k(g_k(B)) = \psi_k(g_k(B'))$ we also have that $b_1' \leq b_i'$ and $d_1' \geq d_i'$ for all $i\in [n]$.
	Now, let $\alpha = \frac{d_1-b_1}{d_1'-b_1'}$, $\delta = b_1 -\alpha b_1'$, and define $T: \mathbb{R} \to \mathbb{R}$ be function $T(x) = \alpha x +\delta$. Observe that $T(b_1')  = b_1$, and that $T(d_1') =d_1$. Now let $(b_i, d_i)$ be another bar and let $m$ denote the number of $1$'s that appear before the first $i$ in $\psi_k(g_k(B))$, or equivalently in $\psi_k(g_k(B'))$. Note $0 < m < 2^{k}$ since $b_1 < b_i < d_1$. 	Then we have that,
\begin{alignat*}{2}
    b_1 +(m-1)\frac{d_1-b_1}{2^k} &< \, b_i\, &&< b_1 +m\frac{(d_1-b_1)}{2^k}, \text{ and} \\
    b_1' +(m-1)\frac{d_1'-b_1'}{2^k} &<\, b_i'\, &&< b_1' +m\frac{(d_1'-b_1')}{2^k},
\end{alignat*}
and since $\alpha >0$, it follows that
$$
b_1 +\frac{(m-1)(d_1-b_1)}{2^k} <\, T(b_i') \,< b_1 +\frac{m(d_1-b_1)}{2^k}.
$$
Therefore, $\lvert b_i - T(b_i')\rvert\leq \frac{d_1-b_1}{2^k}$. Recall that $d_1 \geq d_i$ for all $i\in [n]$, so by repeating the argument we get that $\lvert d_i - T(d_i')\rvert\leq \frac{d_1-b_1}{2^k}$. Thus, $d_\infty(B, \alpha B' +\delta) \leq \frac{\lvert d_*-b_*\rvert}{2^k}$. For the bound on the $q$-Wasserstein distance, observe that:
\begin{alignat*}{2}
	\Big( \sum_{i=1}^n \| (b_i,d_i) - (T(b_i') -T(d_i')\|_\infty^q \Big)^{\frac{1}{q}}
	\leq \Big( \sum_{i=2}^n \big(\frac{d_1-b_1}{2^k}\big)^q \Big)^{\frac{1}{q}}
	&=  \Big( (n-1)(\frac{d_1-b_1}{2^k}\big)^q \Big)^{\frac{1}{q}} \\
	&= (n-1)^{\frac{1}{q}} \frac{d_1-b_1}{2^k},
\end{alignat*}
from which the result follows.
\end{proof}
\begin{remark}
	We note that the extra assumption in Theorem \ref{convergence} is necessary. Consider the barcodes $B = \{(0,1), (1-\epsilon, 1+\epsilon)\}$ and $B' = \{(0,1), (1-\epsilon, 2)\}$, where $\epsilon << \frac{1}{2^k}$.  Let $T(x) = \alpha x + \delta$, with $\alpha = \frac{d_1-b_1}{d_1'-b_1'}$ and $\delta = b_1 -\alpha b_1'$ as in Theorems \ref{affine} and \ref{convergence}. Note $T$ is simply the identity map, so $\lvert T(2) - 1+\epsilon\rvert = 1-\epsilon$. Hence $d_\infty(B, \alpha B' +\delta) = 1-\epsilon$. Thus, for fixed $\lvert d_*-b_*\rvert$ and arbitrary $k$ we can find a barcode $B'$ such that $\psi_k(g_k(B)) = \psi_k(g_k(B'))$ and $d_\infty(B, \alpha B' +\delta)$ is arbitrarily close to 1.
\end{remark}

\subsection{Barcode Polytopes}
It is a well known result that the permutahedron $\mathfrak{S}_n$ is also the face lattice of the polytope $P_{\mathfrak{S}_n} = \conv\{(\pi_1, \ldots, \pi_n) \in \mathbb{R}^n : \pi \in \mathfrak{S_n}\}$. Recall that by composing the maps $\iota \circ \psi_k$ we get an embedding of $L((2^k+1)^n)/\mathfrak{S}_{n}$ into $P_{\mathfrak{S}_{n(2^k+1)}}$. This gives us a new polytope, $P_{n,k} = \conv\{ (\pi_1, \ldots, \pi_{n(2^k+1)}) \in \mathbb{R}^{n(2^k+1)} : \pi \in \image(\iota \circ \psi_k)\}$. We call $P_{n,k}$ the \emph{power-$k$ barcode polytope}.

Because this embedding sends $L((2^k+1)^n)/\mathfrak{S}_{n}$ to a prime-ideal, the polytope $P_{n,k}$ is an example of a \emph{Bruhat interval polytope} \cite{Williams2015}.
\begin{definition}
	Let $u \leq v$ be permutations in $\mathfrak{S}_n$. The Bruhat interval polytope $Q_{u,v}$ is the convex hull of all permutation vectors $(z_1, z_2, ..., z_n)$ with $u \leq z \leq v$.
\end{definition}
Note that $P_{n,k}$ is equal to $Q_{u,v}$ for $u = e \in \mathfrak{S}_{n(2k+1)}$ and $v$ the “fully nested” permutation $(1_1~2_1~\ldots~n_1~n_2~\ldots~n_{2k+1}~(n-1)_2~\ldots~1_{2^k}~1_{2^k +1})$, where
again we identify $\mathfrak{S}_{n(2k+1)}$ with permutations of the totally ordered set $1_1 < 1_2 <\ldots n_1 <\ldots<n_{2^k+1}$.

In \cite{Williams2015}, the authors prove, among other things, the following formula for computing the dimension of a Bruhat interval polytope.
	Let $u \leq v$ be permutations in $\mathfrak{S}_n$, and let $C : u = x_0 \lessdot  x_1 \lessdot \ldots \lessdot x_\ell = v$ be any maximal chain from $u$ to $v$.
	Define a labeled graph $G^C$ on $[n]$ having an edge between vertices $a$ and $b$ if and only if $x_i(ab) = x_{i+1}$ for some $0 \leq i \leq \ell - 1$. Define $\Pi_C = {V_1, V_2, ..., V_r}$ to be the partition of $[n]$  whose blocks $V_j$ are the connected components of $G^C$.The authors show that the number of blocks does not depend on the choice of maximal chain $C$, so we let $\#\Pi_{u,v}$ denote the number of blocks, $r$. The authors then prove the following.
	
\begin{theorem}[\cite{Williams2015}]
    The dimension of the Bruhat interval polytope $Q_{u,v}$ is $(n -\#\Pi_{u,v}).$
\end{theorem}
From this result, it is easy to compute the dimension of the barcode
polytopes $P_{n,k}$.

\begin{cor}\label{polytope_dim}
The dimension of the power-$k$ barcode polytope, $P_{n,k}$ is $n(2^k+1) -2$.
\end{cor}

\begin{proof}
    Recall, $P_{n,k} = Q_{u,v}$ for $u = e \in \mathfrak{S}_{n(2k+1)}$ and $v$ the “fully nested” permutation $(1_1~2_1~\ldots~n_1~n_2~\ldots~n_{2k+1}~(n-1)_2~\ldots~1_{2^k}~1_{2^k +1})$. Consider the maximal chain $C$ from $u$ to $v$ given by moving each element into position one by one, starting with the 1’s in descending lexicographic order, then the 2’s is descending order, and so on. Note that traversing this chain requires that we use all adjacent transpositions except $(12)$ since the $1_1$ term does not move. Hence, $\Pi_C = \{\{1_1\}, \{1_2,\ldots, 1_{2k+1}, 2_1, \ldots,n_{2k+1}\}$. Therefore $\Pi_{u,v} = 2$, from which the result follows.
\end{proof}

\section{Conclusion and Future Work}\label{conclusion}
In this paper we developed a new family of combinatorial invariants of barcodes by mapping barcodes into various equivalence classes of multipermutations. We then studied the poset structure of these equivalence classes. We showed that the rank of one such poset has an elegant interpretation in terms of a crossing number for barcodes. Moreover, we show that these posets in fact lattices, specifically the face lattices of interval polytopes in the classic permutahedron. 
Finally, we showed that these multipermutations can provide bounds on the bottleneck and Wasserstein distances for a large class for barcodes. In this way, they are discrete signatures that reflect continuous information.

In \cite{kanari2020trees}, the authors developed their own combinatorial invariant, which are permutations in $\mathfrak{S}_n$, and study the number of combinatorial trees whose barcodes can produce such an invariant. It would be interesting to study similar “inverse counting problems” for our invariants. At this time it not clear which types of filtrations or simplicial complexes give rise to barcodes that will produce nice inverse counting formulae. 

The barcodes lattices are also of interest purely from the perspective of enumerative combinatorics. For instance, at this time we do not have a description of their rank generating functions, M\"{o}bius functions, number of maximal chains, or other classical results on posets.
\linebreak

\noindent \textbf{Ackowledgements} The author gratefully acknowledges partial support from NSF DMS-grant 1818969, NSF TRIPODS Award no. CCF-1934568, and from the NSF-AGEP supplement.

\bibliographystyle{plain}
\bibliography{references__1}

\end{document}